\documentclass[a4paper, 12pt]{amsart}

\usepackage{amsmath}
\usepackage{amssymb}
\usepackage{amsthm}
\usepackage{graphicx}
\usepackage[dvipdfmx]{hyperref}

\makeatletter

\@addtoreset{equation}{section}
\makeatother 

\theoremstyle{plain}
\newtheorem{thm}{Theorem}[section]
\newtheorem*{thm*}{Theorem}
\newtheorem{prop}[thm]{Proposition}
\newtheorem{lem}[thm]{Lemma}

\theoremstyle{definition}
\newtheorem{ex}[thm]{Example}%[section]

\theoremstyle{remark}
\newtheorem{rem}[thm]{Remark}%[section]

\renewcommand{\epsilon}{\varepsilon}

\newcommand{\diam}{\operatorname{diam}}

\newcommand{\Lip}{\mathrm{Lip}}

\newcommand{\rarrow}{\rightarrow}
\newcommand{\orarrow}{\overrightarrow}

\newcommand{\olarrow}{\overleftarrow}

\usepackage{latexsym}

\title[Maximal diameter theorem for directed graphs]{Maximal diameter theorem for directed graphs of positive Ricci curvature}

\author{Ryunosuke Ozawa}
\address{Department of Mathematics, National Defense Academy of Japan, 1-10-20 Hashirimizu, Yokosuka, 239-8686, Japan}
\email{rozawa@nda.ac.jp}

\author{Yohei Sakurai}
\address{Department of Mathematics, Saitama University, 255 Shimo-Okubo, Sakura-ku, Saitama-City, Saitama, 338-8570, Japan}
\email{ysakurai@rimath.saitama-u.ac.jp}

\author{Taiki Yamada}
\address{Shimane University, 1060 Nishikawatsu-cho, Matsue, Shimane, 690-8504, Japan}
\email{taiki\_yamada@riko.shimane-u.ac.jp}

\date{April 22, 2021}

\begin{document}
\maketitle

\begin{abstract}
In a previous work,
the authors \cite{OSY} have introduced a Lin-Lu-Yau type Ricci curvature for directed graphs,
and obtained a diameter comparison of Bonnet-Myers type.
In this paper,
we investigate rigidity properties for the equality case,
and conclude a maximal diameter theorem of Cheng type.
\end{abstract}

%%%%%%%%%%%%%%%%%%%%%%%%%%%%
%%%%%%%%%%%%%%%%%%%%%%%%%%%%
%%%%%%%%%%%%%%%%%%%%%%%%%%%%
\section{Introduction}

For a Riemannian manifold of positive Ricci curvature,
the Bonnet-Myers theorem asserts that
its diameter is bounded from above by that of a corresponding standard sphere.
Moreover,
the Cheng maximal diameter theorem says that
the equality in Bonnet-Myers theorem holds if and only if it is isometric the sphere.
In this article,
we are concerned with their discrete analogues.
In \cite{OSY},
the authors have introduced a Lin-Lu-Yau type Ricci curvature for directed graphs,
and produced a Bonnet-Myers type diameter comparison theorem.
We now aim to examine its equality case.
In that case,
we derive several results on geometric structure,
and some analytic results for the Chung Laplacian.

%%%%%%%%%%%%%%%%%%%%%%%%%%%%
\subsection{Main results}

Recently,
there are some attempts to introduce the notion of the Ricci curvature for discrete spaces,
and the Ricci curvature for (undirected) graphs introduced by Lin-Lu-Yau \cite{LLY} is one of the well-studied objects (see also a pioneering work of Ollivier \cite{O}).
It is well-known that
a lower Ricci curvature bound of them leads us to various geometric and analytic consequences (see e.g., \cite{AS}, \cite{BJL}, \cite{BRT}, \cite{CKKLMP}, \cite{JL}, \cite{LLY}, \cite{MW}, \cite{P}, and so on).
In \cite{LLY},
they have provided a discrete analogue of Bonnet-Myers theorem in Riemannian geometry.
Furthermore,
Cushing et al. \cite{CKKLMP} have established that of Cheng maximal diameter theorem for regular graphs.
They have investigated their geometric structure in the equality case,
and concluded a classification of self-centered regular graphs based on \cite{KMS}.

In \cite{OSY},
the authors have generalized the Lin-Lu-Yau Ricci curvature for directed graphs referring to the formulation of the Laplacian by Chung \cite{C1}, \cite{C2},
and extended the Bonnet-Myers theorem in \cite{LLY} to the directed setting.
The purpose of this paper is to observe rigidity phenomena in the equality case.
In order to state our main results,
we briefly recall some notions on directed graphs (more precisely, see Sections \ref{sec:Preliminaries} and \ref{sec:Curvatures}).
Let $(V,\mu)$ denote a simple,
strongly connected,
finite weighted directed graph,
where $V$ denotes the vertex set,
and $\mu:V\times V\to [0,\infty)$ is the (non-symmetric) edge weight.
We denote by $d:V\times V\to [0,\infty)$ the (non-symmetric) distance function on $V$.
For $x,y\in V$ with $x\neq y$,
let $\kappa(x,y)$ stand for the \textit{Ricci curvature} introduced in \cite{OSY} (see Subsection \ref{sec:Ricci curvature}),
and set
\begin{equation}\label{eq:infconst}
K:=\inf_{x\rarrow y}\kappa(x,y),
\end{equation}
where the infimum is taken over all $x,y\in V$ with $x\rarrow y$.
Here $x\rarrow y$ means that
there exists a directed edge from $x$ to $y$.
Also,
for $x,y\in V$,
let $\mathcal{H}(x,y)$ denote the \textit{mixed asymptotic mean curvature} introduced in \cite{OSY} (see Subsection \ref{sec:Mean curvature}),
and set
\begin{equation}\label{eq:supconst}
\Lambda:=\sup_{x,y\in V}\mathcal{H}(x,y).
\end{equation}
Notice that
$\Lambda \geq 2$ in general.
The diameter comparison in \cite{OSY} can be stated as follows (see \cite[Theorem 8.3]{OSY}, and also Theorem \ref{thm:LLY diameter comparison} below):
\begin{thm}[\cite{OSY}]\label{thm:BM}
Let $(V,\mu)$ denote a simple,
strongly connected,
finite weighted directed graph.
If $K>0$,
then we have
\begin{equation}\label{eq:BM}
\diam V\leq \frac{\Lambda}{K}.
\end{equation}
\end{thm}

We study the equality case in Theorem \ref{thm:BM}.
For $x,y\in V$ with $x\neq y$,
we say that 
$(V,\mu)$ is \textit{spherically suspended with poles $(x,y)$} if the following properties hold:
\begin{enumerate}\setlength{\itemsep}{+0.8mm}
\item $V$ is covered by minimal geodesics from $x$ to $y$, namely,
\begin{equation*}
V=\{z\in V \mid d(x,z)+d(z,y)=d(x,y)\};
\end{equation*}
\item for every minimal geodesic $\{x_{i}\}^{d(x,y)}_{i=0}$ from $x$ to $y$,
it holds that $\kappa(x_{i},x_{j})=K$ for all $i,j \in \{0,\dots,d(x,y)\}$ with $i < j$;
\item $\mathcal{H}(x,y)=\Lambda$.
\end{enumerate}
One of our main theorem is the following structure theorem:
\begin{thm}\label{thm:Cheng}
Let $(V,\mu)$ denote a simple,
strongly connected,
finite weighted directed graph.
We assume $K>0$, and the equality in \eqref{eq:BM} holds, namely,
\begin{equation}\label{eq:equalBM}
\diam V= \frac{\Lambda}{K}.
\end{equation}
Then $(V,\mu)$ is spherically suspended with poles $(x,y)$ for any $x,y\in V$ with $d(x,y)=\diam V$.
\end{thm}

Cushing et al. \cite{CKKLMP} have proved Theorem \ref{thm:Cheng} for unweighted, undirected regular graphs (see \cite[Lemma 5.3 and Theorem 5.5]{CKKLMP}).

\begin{rem}
It seems that the method of the proof in \cite{CKKLMP} works only for regular graphs.
Actually,
it is based on the characterization result of the Lin-Lu-Yau Ricci curvature by the so-called Ollivier Ricci curvature with idleness parameter,
which has been obtained in \cite{BCLMP} (see the proof of \cite[Lemma 5.3]{CKKLMP}, and cf. \cite[Remark 2.3]{CKKLMP}).
We need to reconsider the method of the proof that is suitable for our setting.
To overcome this issue,
we refer to the proof of the Cheng maximal diameter theorem in the smooth setting.
Here we recall that
the Cheng maximal diameter can be proved by combining the Laplacian comparison theorem for the distance functions from poles, which leads us to the superharmonicity of the sum of them, and the minimum principle.
We will prove Theorem \ref{thm:Cheng} by verifying the minimum principle in our setting, and analyzing our Laplacian comparison (see Lemmas \ref{lem:minimum principle} and \ref{lem:harmonic} below).
\end{rem}

\begin{rem}
Matsumoto \cite{M} has stated Theorem \ref{thm:Cheng} for unweighted, undirected (not necessarily regular) graphs in his master thesis in 2010.
His method of the proof is based on primitive mass transport techniques without analytic argument,
which is completely different from our method.
\end{rem}

%%%%%%%%%%%%%%%%%%%%%%%%
\subsection{Organization}\label{sec:Organization}
In Section \ref{sec:Preliminaries},
we will review basics of directed graphs.
In Section \ref{sec:Curvatures},
we recall the formulation of the Ricci curvature introduced in \cite{OSY},
and examine its properties.
In Section \ref{sec:Proof},
we prove Theorem \ref{thm:Cheng}.
We also show that the equalities in other comparison geometric results hold under the same setting as in Theorem \ref{thm:Cheng} (see Theorems \ref{thm:rigidLapcomp} and \ref{thm:rigideigcomp}).
In Section \ref{sec:Examples},
we will present some examples having maximal diameter.

%%%%%%%%%%%%%%%%%%%%%%%%%%%%
%%%%%%%%%%%%%%%%%%%%%%%%%%%%
%%%%%%%%%%%%%%%%%%%%%%%%%%%%
\section{Preliminaries}\label{sec:Preliminaries}
We here review basics and fundamental facts on directed graphs.
We refer to \cite{OSY}.

%%%%%%%%%%%%%%%%%%%%%%%%
\subsection{Directed graphs}\label{sec:Directed graphs}
Let $(G,\mu)$ be a finite weighted directed graph,
namely,
$G=(V,E)$ is a finite directed graph,
and $\mu:V \times V\to [0,\infty)$ is a function such that $\mu(x,y)>0$ if and only if $x\rightarrow y$,
where $x\rightarrow y$ means $(x,y) \in E$ as stated in the above section.
We will denote by $n$ the cardinality of $V$.
The function $\mu$ is called the \textit{edge weight},
and we write $\mu(x,y)$ by $\mu_{xy}$.
Note that
$(G,\mu)$ is undirected if and only if $\mu_{xy}=\mu_{yx}$ for all $x,y\in V$,
and simple if and only if $\mu_{xx}=0$ for all $x\in V$.
Also,
it is called \textit{unweighted} if $\mu_{xy}=1$ whenever $x\rightarrow y$.
The weighted directed graph can be written as $(V,\mu)$ since the full information of $E$ is included in $\mu$.
We use $(V,\mu)$ instead of $(G,\mu)$ as in \cite{OSY}. 

For $x \in V$, 
its \textit{outer neighborhood $N_{x}$}, \textit{inner one $\olarrow{N}_{x}$}, and \textit{neighborhood $\mathcal{N}_{x}$} are defined as
\begin{equation*}\label{eq:neighborhoods}
N_{x}:= \left\{y \in V \mid x\rightarrow y \right\},\quad \olarrow{N}_{x}:= \left\{y \in V \mid y \rightarrow x \right\},\quad \mathcal{N}_{x}:=N_{x} \cup \olarrow{N}_{x},
\end{equation*}
respectively.
The \textit{outer degree $\orarrow{\deg}(x)$} and \textit{inner one $\olarrow{\deg}(x)$} are defined as the cardinality of $N_{x}$ and $\olarrow{N}_{x}$,
respectively.
In the unweighted case,
$(V,\mu)$ is called \textit{Eulerian} if $\orarrow{\deg}(x)=\olarrow{\deg}(x)$ for all $x\in V$.

For $x,y\in V$,
a sequence $\left\{x_{i} \right\}_{i=0}^{l}$ of vertices is said to be \textit{directed path} from $x$ to $y$ if $x_0=x,\,x_{l}=y$ and $x_{i}\rarrow x_{i+1}$ for all $i=0,\dots,l-1$,
where $l$ is called its length.
Further,
$(V,\mu)$ is called \textit{strongly connected} if
for all $x,y\in V$,
there exists a directed path from $x$ to $y$.
For strongly connected $(V,\mu)$,
the (non-symmetric) distance function $d:V\times V\to [0,\infty)$ is defined as follows:
$d(x,y)$ is defined to be the minimum of the length of directed paths from $x$ to $y$.
A directed path $\left\{x_{i} \right\}_{i=0}^{l}$ from $x$ to $y$ is called \textit{minimal geodesic} when $l=d(x,y)$.
The \textit{diameter} of $(V,\mu)$ is defined as
\begin{equation*}\label{eq:diam}
\diam V:=\sup_{x,y\in V}d(x,y).
\end{equation*}
For $x\in V$,
the \textit{distance function $\rho_{x}:V\to \mathbb{R}$},
and the \textit{reverse one $\olarrow{\rho}_{x}:V\to \mathbb{R}$} are done as
\begin{equation*}\label{eq:distance function from a single point}
\rho_{x}(y):=d(x,y),\quad \olarrow{\rho}_{x}(y):=d(y,x).
\end{equation*}
For $L>0$,
a function $f:V\to \mathbb{R}$ is said to be \textit{$L$-Lipschitz} if
\begin{equation*}
f(y)- f(x) \leq L\, d(x,y)
\end{equation*}
for all $x,y\in V$.
Note that
$\rho_{x}$ is $1$-Lipschitz,
but $\olarrow{\rho}_{x}$ is not always $1$-Lipschitz.
Let $\Lip_{L}(V)$ be the set of all $L$-Lipschitz functions on $V$.

%%%%%%%%%%%%%%%%%%%%%%%%
\subsection{Laplacian}\label{sec:Laplacians}
In what follows,
let $(V,\mu)$ be a simple,
strongly connected,
finite weighted directed graph.
In this subsection,
we recall the formulation of the Chung Laplacian introduced in \cite{C1}, \cite{C2}.
The \textit{transition probability kernel} $P:V\times V\to [0,1]$ is defined as
\begin{equation*}\label{eq:Markov kernel}
P(x,y):=\frac{\mu_{xy}}{\mu(x)},
\end{equation*}
where
\begin{equation*}
\mu(x):=\sum_{y\in V}\mu_{xy}.
\end{equation*}
Since $(V,\mu)$ is finite and strongly connected,
the Perron-Frobenius theorem guarantees that
there exists a unique (up to scaling) positive function $m:V\to (0,\infty)$ such that
\begin{equation}\label{eq:Perron Frobenius}
m(x)=\sum_{y\in V}m(y)P(y,x).
\end{equation}
A probability measure $\mathfrak{m}:V\to (0,1]$ satisfying (\ref{eq:Perron Frobenius}) is called the \textit{Perron measure}.

Let $\mathfrak{m}$ stand for the Perron measure.
We define the \textit{reverse transition probability kernel} $\olarrow{P}:V\times V\to [0,1]$,
and the \textit{mean transition probability kernel} $\mathcal{P}:V\times V\to [0,1]$ by
\begin{equation*}\label{eq:perron vector}
\olarrow{P}(x,y):=\frac{\mathfrak{m}(y)}{\mathfrak{m}(x)}P(y,x),\quad \mathcal{P}:=\frac{1}{2}(P+\olarrow{P}).
\end{equation*}

\begin{rem}\label{rem:positive}
We see that $\mathcal{P}(x,y)>0$ if and only if $y\in \mathcal{N}_x$.
\end{rem}

\begin{rem}\label{rem:Eulerian}
When $(V,\mu)$ is Eulerian,
we possess the following expression (see \cite[Examples 1, 2, 3]{C1} and \cite[Remarks 2.2 and 2.3]{OSY})
\begin{equation}\label{eq:Euler}
\mathcal{P}(x,y)= \begin{cases}
		\cfrac{1}{\orarrow{\deg}(x)} & \text{if $x\rightarrow y$ and $y\rightarrow x$},\\
		\cfrac{1}{2\,\orarrow{\deg}(x)} & \text{if either $x\rightarrow y$ or $y\rightarrow x$},\\
		0 & \text{otherwise}.
		\end{cases}
\end{equation}
\end{rem}

Let $\mathcal{F}$ be the set of all functions on $V$.
The \textit{Chung Laplacian} $\mathcal{L}:\mathcal{F}\to \mathcal{F}$ is formulated by
\begin{equation*}\label{eq:Chung Laplacian}
\mathcal{L}f(x):=f(x)-\sum_{y\in V} \mathcal{P}(x,y)f(y),
\end{equation*}
which is symmetric with respect to $\mathfrak{m}$.

%%%%%%%%%%%%%%%%%%%%%%%%
\subsection{Optimal transport theory}\label{sec:OT}

We review the basics of the optimal transport theory (cf. \cite{V2}).
For two probability measures $\nu_{0},\nu_{1}$ on $V$,
a probability measure $\pi :V\times V\to [0,\infty)$ is called a \textit{coupling of $(\nu_{0},\nu_{1})$} if
\begin{equation*}
\sum_{y \in V}\pi(x,y) =\nu_{0}(x),\quad \sum_{x \in V}\pi(x,y) = \nu_{1}(y).
\end{equation*}
Let $\Pi(\nu_{0},\nu_{1})$ be the set of all couplings of $(\nu_{0},\nu_{1})$.
The \textit{Wasserstein distance} from $\nu_{0}$ to $\nu_{1}$ is defined as
\begin{equation}\label{eq:Wasserstein distance}
W(\nu_{0},\nu_{1}):=\inf_{\pi \in \Pi(\nu_{0},\nu_{1})} \sum_{x,y \in V}d(x,y)\pi(x,y),
\end{equation}
which is a (non-symmetric) distance function on the set of all probability measures on $V$.

The following Kantorovich-Rubinstein duality formula is well-known (cf. \cite[Theorem 5.10 and Particular Cases 5.4 and 5.16]{V2}):
\begin{prop}\label{prop:Kantorovich duality}
For any two probability measures $\nu_{0}, \nu_{1}$ on $V$, we have
\begin{equation*}
W(\nu_{0},\nu_{1}) = \sup_{f\in \Lip_{1}(V)} \sum_{x\in V} f(x)\left(  \nu_{1}(x)-\nu_{0}(x)  \right).
\end{equation*}
\end{prop}

%%%%%%%%%%%%%%%%%%%%%%%%
%%%%%%%%%%%%%%%%%%%%%%%%
%%%%%%%%%%%%%%%%%%%%%%%%
\section{Curvatures}\label{sec:Curvatures}

In this section,
we investigate some basic properties of curvatures on directed graphs.

%%%%%%%%%%%%%%%%%%%%%%%%
\subsection{Ricci curvature}\label{sec:Ricci curvature}
For $\epsilon \in [0,1]$,
and for $x,y\in V$ with $x\neq y$,
we set
\begin{equation*}\label{eq:pre Ricci curvature}
\kappa_{\epsilon}(x,y):=1-\frac{W(\nu^{\epsilon}_{x},\nu^{\epsilon}_{y})}{d(x,y)},
\end{equation*}
where $\nu^{\epsilon}_{x}:V\to [0,1]$ is a probability measure defined by
\begin{equation*}
\nu^{\epsilon}_{x}(z)=(1-\epsilon)\delta_{x}(z)+\epsilon \,\mathcal{P}(x,z)
\end{equation*}
for the Dirac measure $\delta_x$.
The authors \cite{OSY} have introduced the \textit{Ricci curvature} as follows (see \cite[Definition 3.6]{OSY}):
\begin{equation*}
\kappa(x,y):=\lim_{\epsilon\to 0}\frac{\kappa_{\epsilon}(x,y)}{\epsilon},
\end{equation*}
which is well-defined (see \cite[Lemmas 3.2 and 3.4, and Definition 3.6]{OSY}).
In the undirected case,
this is nothing but the Lin-Lu-Yau Ricci curvature in \cite{LLY}.

We recall a characterization of the Ricci curvature in terms of the Chung Laplacian $\mathcal{L}$.
Let $x,y\in V$ with $x\neq y$.
The \textit{gradient operator} is defined by
\begin{equation*}\label{eq:gradient operator}
\nabla_{xy} f:=\frac{f(y)-f(x)}{d(x,y)}
\end{equation*}
for $f:V\to \mathbb{R}$.
We set
\begin{equation*}
\mathcal{F}_{xy}:=\{f\in \Lip_{1}(V)  \mid \nabla_{xy}f=1\}.
\end{equation*}
We possess the following characterization,
which has been obtained by M\"unch-Wojciechowski \cite{MW} in the undirected case (see \cite[Theorem 3.10]{OSY}, and also \cite[Theorem 2.1]{MW}):
\begin{thm}[\cite{MW}, \cite{OSY}]\label{thm:limit free formula}
\begin{equation*}
\kappa(x,y)=\inf_{f\in \mathcal{F}_{xy}}  \nabla_{xy} \mathcal{L} f.
\end{equation*}
\end{thm}

Our Ricci curvature satisfies the following (see \cite[Lemma 2.3]{LLY}, \cite[Proposition 3.8]{OSY}):
\begin{prop}[\cite{LLY}, \cite{OSY}]\label{prop:local to global}
\begin{equation*}
\inf_{x\neq y}\kappa(x,y)\geq K,
\end{equation*}
where $K$ is defined as \eqref{eq:infconst}.
\end{prop}
The authors \cite{OSY} have stated Proposition \ref{prop:local to global} without proof.
For completeness,
we give its proof.
Proposition \ref{prop:local to global} is a direct consequence of the following lemma:
\begin{lem}\label{lem:chain}
Let $x,y\in V$ with $x \neq y$,
and let $\{x_{i}\}^{l}_{i=0}$ be a minimal geodesic from $x$ to $y$ with $l=d(x,y)$.
We set $z:=x_{a}$ and $w:=x_{b}$ for some $a,b \in \{0,\dots,l\}$ with $a < b$.
Then the following hold:
\begin{enumerate}
\item If $a\geq 1$ and $b\leq l-1$,
then
\begin{equation}\label{eq:chain1}
\kappa(x,y) d(x,y)\geq \sum^{a-1}_{i=0}\kappa(x_{i},x_{i+1})+\kappa(z,w)d(z,w)+\sum^{l-1}_{i=b}\kappa(x_{i},x_{i+1});
\end{equation}
\item If $a\geq 1$ and $b=l$,
then
\begin{equation}\label{eq:chain2}
\kappa(x,y) d(x,y)\geq \sum^{a-1}_{i=0}\kappa(x_{i},x_{i+1})+\kappa(z,y)d(z,y);
\end{equation}
\item If $a=0$ and $b\leq l-1$,
then
\begin{equation}\label{eq:chain3}
\kappa(x,y) d(x,y)\geq \kappa(x,w)d(x,w)+\sum^{l-1}_{i=b}\kappa(x_{i},x_{i+1}).
\end{equation}
\end{enumerate}
\end{lem}
\begin{proof}
We only show the inequality \eqref{eq:chain1}.
The others \eqref{eq:chain2}, \eqref{eq:chain3} can be proved by the same argument,
and more easily.
By the triangle inequality of $W$,
\begin{equation*}
W(\nu^{\epsilon}_{x},\nu^{\epsilon}_{y})\leq \sum^{a-1}_{i=0}W(\nu^{\epsilon}_{x_{i}},\nu^{\epsilon}_{x_{i+1}})
+W(\nu^{\epsilon}_{z},\nu^{\epsilon}_{w})
+\sum^{l-1}_{i=b}W(\nu^{\epsilon}_{x_{i}},\nu^{\epsilon}_{x_{i+1}}).
\end{equation*}
Since $d(x,y)=l$ and $d(z,w)=b-a$, we see
\begin{align*}
&\quad\,\,\kappa_{\epsilon}(x,y)\\
&\geq 1-\frac{1}{l}\left\{  \sum^{a-1}_{i=0}W(\nu^{\epsilon}_{x_{i}},\nu^{\epsilon}_{x_{i+1}})
+W(\nu^{\epsilon}_{z},\nu^{\epsilon}_{w})
+\sum^{l-1}_{i=b}W(\nu^{\epsilon}_{x_{i}},\nu^{\epsilon}_{x_{i+1}}) \right\}\\
&\geq 1+\frac{1}{l}\left\{-a + \sum^{a-1}_{i=0}(1-W(\nu^{\epsilon}_{x_{i}},\nu^{\epsilon}_{x_{i+1}}))
 \right\}
+ \frac{1}{l}\left\{ -d(z,w)+ d(z,w) \left(  1-\frac{W(\nu^{\epsilon}_{z},\nu^{\epsilon}_{w})}{d(z,w)}   \right)  \right\}\\
&\qquad \qquad \qquad \qquad \qquad \qquad \qquad \qquad \quad \,\,\,\,\, +\frac{1}{l}\left\{-(l-b) + \sum^{l-1}_{i=b}(1-W(\nu^{\epsilon}_{x_{i}},\nu^{\epsilon}_{x_{i+1}}))\right\}\\
&= \frac{1}{l}\left\{\sum^{a-1}_{i=0}\kappa_{\epsilon}(x_i,x_{i+1})+(b-a) \kappa_{\epsilon}(z,w)+\sum^{l-1}_{i=b}\kappa_{\epsilon}(x_i,x_{i+1}) \right\}.
\end{align*}
By dividing the both sides by $\epsilon$,
and by letting $\epsilon \to 0$,
we complete the proof.
\end{proof}

\begin{proof}[Proof of Proposition \ref{prop:local to global}]
Proposition \ref{prop:local to global} follows from choosing $b=a+1$ in Lemma \ref{lem:chain}.
\end{proof}

%%%%%%%%%%%%%%%%%%%%%%%%
\subsection{Asymptotic mean curvature}\label{sec:Mean curvature}
In the present subsection,
we recall the notion of the asymptotic mean curvature introduced by the authors \cite{OSY}.
For $x\in V$,
the \textit{asymptotic mean curvature $\mathcal{H}_{x}$ around $x$},
and the \textit{reverse one $\olarrow{\mathcal{H}}_{x}$} are defined as
\begin{equation*}\label{eq:asymptotic mean curvature}
\mathcal{H}_{x}:=\mathcal{L} \rho_{x}(x),\quad \olarrow{\mathcal{H}}_{x}:=\mathcal{L} \olarrow{\rho}_{x}(x).
\end{equation*}
It holds that $\mathcal{H}_{x}\leq -1$ and $\olarrow{\mathcal{H}}_{x}\leq -1$ in general;
moreover, the equalities hold in the undirected case.
In particular,
they play an essential role in the directed case.
For $x,y\in V$,
the \textit{mixed asymptotic mean curvature} $\mathcal{H}(x,y)$ is defined by
\begin{equation*}\label{eq:mixed asymptotic mean curvature}
\mathcal{H}(x,y):=-(\mathcal{H}_{x}+\olarrow{\mathcal{H}}_{y}).
\end{equation*}
We have $\mathcal{H}(x,y)\geq 2$;
moreover,
the equality holds in the undirected case.

%%%%%%%%%%%%%%%%%%%%%%%%%%%%
\subsection{Products}\label{sec:Products}
We consider the weighted Cartesian product of $(V,\mu)$,
and another simple, strongly connected, finite weighted directed graph $(V',\mu')$.
For two parameters $\alpha, \beta>0$,
the \textit{$(\alpha,\beta)$-weighted Cartesian product $(V, \mu) \square_{(\alpha,\beta)} (V', \mu')$ of $(V,\mu)$ and $(V',\mu')$} is defined as follows:
Its vertex set is $V\times V'$,
and its edge weight is
\begin{equation*}\label{eq:product weight}
\mu_{(\alpha,\beta)}(\mathbf{x},\mathbf{y}):=\beta\mu'(x')\mu_{xy}\,\delta_{x'}(y')+\alpha\mu(x)\mu'_{x'y'}\,\delta_x(y)
\end{equation*}
for $\mathbf{x}=(x,x'), \mathbf{y}=(y,y') \in V\times V'$,
where $\mu'(x')$ denotes the vertex weight at $x'$ on $(V',\mu')$.
The distance function $d_{(\alpha,\beta)}:(V\times V')\times (V\times V')\to [0,\infty)$ should be
\begin{equation}\label{eq:product distance}
d_{(\alpha,\beta)}(\mathbf{x},\mathbf{y})=d(x,y)+d'(x',y')
\end{equation}
for the distance functions $d$ and $d'$ on $(V,\mu)$ and $(V',\mu')$,
respectively.

Let $\mathbf{x}=(x,x'), \mathbf{y}=(y,y') \in V\times V'$.
Let $\mathcal{H}'_{x'},\olarrow{\mathcal{H}}'_{x'},\mathcal{H}'(x',y')$ denote
the asymptotic mean curvature around $x'$,
the reverse one,
the mixed asymptotic mean curvature over $(V',\mu')$,
respectively.
Also,
let $\mathcal{H}_{(\alpha,\beta),\mathbf{x}},\olarrow{\mathcal{H}}_{(\alpha,\beta),\mathbf{x}},\mathcal{H}_{(\alpha,\beta)}(\mathbf{x},\mathbf{y})$ be
the asymptotic mean curvature around $\mathbf{x}$,
the reverse one,
the mixed asymptotic mean curvature over $(V, \mu) \square_{(\alpha,\beta)} (V', \mu')$,
respectively.
We summarize formulas for asymptotic mean curvature (see \cite[Proposition 5.5]{OSY}):
\begin{prop}[\cite{OSY}]\label{prop:product mean curv}
\begin{align}\notag
\mathcal{H}_{(\alpha,\beta),\mathbf{x}}&=\frac{\beta}{\alpha+\beta}\mathcal{H}_{x}+\frac{\alpha}{\alpha+\beta}\mathcal{H}'_{x'},\\ \notag
\olarrow{\mathcal{H}}_{(\alpha,\beta),\mathbf{x}}&=\frac{\beta}{\alpha+\beta}\olarrow{\mathcal{H}}_{x}+\frac{\alpha}{\alpha+\beta}\olarrow{\mathcal{H}}'_{x'},\\ \label{eq:product mean curv}
\mathcal{H}_{(\alpha,\beta)}(\mathbf{x},\mathbf{y})&=\frac{\beta}{\alpha+\beta}\mathcal{H}(x,y)+\frac{\alpha}{\alpha+\beta}\mathcal{H}'(x',y').
\end{align}
\end{prop}

For $x',y' \in V'$ with $x'\neq y'$,
let $\kappa'(x',y')$ denote the Ricci curvature over $(V',\mu')$.
For $\mathbf{x},\mathbf{y} \in V\times V'$ with $\mathbf{x}\neq \mathbf{y}$,
we further denote by $\kappa_{(\alpha,\beta)}(\mathbf{x},\mathbf{y})$ the Ricci curvature over $(V, \mu) \square_{(\alpha,\beta)} (V', \mu')$.
We possess the following formulas (see \cite[Theorem 5.6]{OSY}, and also \cite[Theorem 3.1]{LLY} in the undirected case):
\begin{thm}[\cite{LLY}, \cite{OSY}]\label{thm:product Ricci}
\begin{enumerate}
\item If $x\neq y$ and $x'\neq y'$,
then
\begin{equation*}\label{eq:product Ricci}
\kappa_{(\alpha,\beta)}(\mathbf{x},\mathbf{y}) = \frac{\beta}{\alpha+\beta} \frac{d(x,y)}{d(x,y) + d'(x',y')} \kappa(x,y)+\frac{\alpha}{\alpha+\beta}\frac{d'(x',y')}{d(x,y) + d'(x',y')} \kappa'(x',y');
\end{equation*}
\item if $x\neq y$ and $x'= y'$,
then
\begin{equation*}\label{eq:pre product Ricci1}
\kappa_{(\alpha,\beta)}(\mathbf{x},\mathbf{y}) = \frac{\beta}{\alpha+\beta}\kappa(x,y);
\end{equation*}
\item if $x= y$ and $x'\neq y'$,
then
\begin{equation*}\label{eq:pre product Ricci2}
\kappa_{(\alpha,\beta)}(\mathbf{x},\mathbf{y}) = \frac{\alpha}{\alpha+\beta} \kappa'(x',y').
\end{equation*}
\end{enumerate}
\end{thm}

%%%%%%%%%%%%%%%%%%%%%%%%%%%%
\subsection{Comparison geometric results}\label{sec:Comparison geometric results}
In this subsection,
we review comparison geometric results under our lower Ricci curvature bound.
We begin with the diameter comparison (see \cite[Theorem 8.3]{OSY}, and also \cite[Theorem 4.1]{LLY} in the undirected case).
\begin{thm}[\cite{LLY}, \cite{OSY}]\label{thm:LLY diameter comparison}
Let $x,y \in V$ with $x \neq y$.
If $\kappa(x,y)>0$, then
\begin{equation*}
d(x,y) \leq \frac{\mathcal{H}(x,y)} {\kappa(x,y)}.
\end{equation*}
\end{thm}

For later convenience,
we explain how to derive Theorem \ref{thm:BM} from Theorem \ref{thm:LLY diameter comparison}.
\begin{proof}[Proof of Theorem \ref{thm:BM}]
Let $x,y\in V$ satisfy $d(x,y)=\diam V$.
By combining Proposition \ref{prop:local to global} and Theorem \ref{thm:LLY diameter comparison},
we conclude
\begin{equation}\label{pf of BM}
K\leq \inf_{z\neq w}\kappa(z,w)\leq \kappa(x,y)\leq \frac{\mathcal{H}(x,y)}{d(x,y)}\leq \frac{\Lambda}{\diam V},
\end{equation}
where $\Lambda$ is defined as \eqref{eq:supconst}.
This completes the proof of Theorem \ref{thm:BM}.
\end{proof}

\begin{rem}\label{rem:pole}
Assume that
the equality in \eqref{eq:BM} holds.
Then the equalities in \eqref{pf of BM} also hold.
In particular,
for any $x,y\in V$ with $d(x,y)=\diam V$,
we see
\begin{equation*}
K= \kappa(x,y)= \frac{\mathcal{H}(x,y)}{d(x,y)}= \frac{\Lambda}{\diam V}.
\end{equation*}
\end{rem}

We next discuss the Laplacian comparison (see \cite[Theorem 1.1]{OSY}, and also \cite[Theorem 4.1]{MW} in the undirected case).
\begin{thm}[\cite{MW}, \cite{OSY}]\label{thm:Laplacian comparison}
Let $x\in V$.
On $V$, we have
\begin{equation*}\label{eq:Laplacian comparison}
\mathcal{L} \rho_{x} \geq K \rho_{x}+\mathcal{H}_x,\quad \mathcal{L}\olarrow{\rho}_x\geq K\olarrow{\rho}_x+\olarrow{\mathcal{H}}_{x}.
\end{equation*}
\end{thm}
\begin{proof}
The first one has been proved in \cite{OSY}.
We show the reverse one.
We fix $y\in V$.
We may assume $y\neq x$.
Notice that $-\olarrow{\rho}_x\in \mathcal{F}_{yx}$.
Thanks to Proposition \ref{prop:local to global} and Theorem \ref{thm:limit free formula},
\begin{equation*}
K \leq \kappa(y,x) \leq \nabla_{yx} \mathcal{L}(-\olarrow{\rho}_x) = \frac{\mathcal{L}(-\olarrow{\rho}_x)(x) - \mathcal{L}(-\olarrow{\rho}_x)(y)}{d(y,x)} = \frac{-\olarrow{\mathcal{H}}_x + \mathcal{L}\olarrow{\rho}_x(y)}{\olarrow{\rho}_x(y)}.
\end{equation*}
This proves the desired one.
\end{proof}

We finally consider the eigenvalue comparison.
Let
\begin{equation*}
0=\lambda_{0}(V)<\lambda_{1}(V)\leq \lambda_{2}(V) \leq \dots \leq \lambda_{n-1}(V)
\end{equation*}
stand for the eigenvalues of $\mathcal{L}$.
We have the following eigenvalue comparison of Lichnerowicz type (see \cite[Theorem 8.2]{OSY}, and also \cite[Theorem 4.2]{LLY} in the undirected case):
\begin{thm}[\cite{LLY}, \cite{OSY}]\label{thm:Lic}
If $K>0$,
then we have
\begin{equation*}
\lambda_{1}(V)\geq K.
\end{equation*}
\end{thm}

%%%%%%%%%%%%%%%%%%%%%%%%%%%%
%%%%%%%%%%%%%%%%%%%%%%%%%%%%
%%%%%%%%%%%%%%%%%%%%%%%%%%%%
\section{Proof of the main results}\label{sec:Proof}

In this section,
we prove our main theorems.

%%%%%%%%%%%%%%%%%%%%%%%%%%%%
\subsection{Structure results}\label{sec:Structure}

A function $f:V\to \mathbb{R}$ is said to be \textit{superharmonic} when
\begin{equation}\label{eq:harmonic}
\mathcal{L}f\geq 0
\end{equation}
over $V$.
We first show the following minimum principle (cf. \cite[Proposition 1.39]{G}):
\begin{lem}\label{lem:minimum principle}
Any superharmonic functions must be constant.
\end{lem}
\begin{proof}
Let $f:V\to \mathbb{R}$ be superharmonic.
Set 
\begin{equation*}
M := \inf_{x \in V} f(x),\quad \Omega:= \{ x \in V \mid f(x) = M \}. 
\end{equation*}
Since $V$ is finite,
$\Omega$ is non-empty.
We first show that
if $x\in \Omega$,
then $\mathcal{N}_{x} \subset \Omega$.
By \eqref{eq:harmonic} and Remark \ref{rem:positive},
we have
\begin{equation*}
M = f(x) \geq \sum_{y \in V} \mathcal{P}(x,y) f(y) = \sum_{y \in \mathcal{N}_{x}} \mathcal{P}(x,y) f(y) \ge  M \sum_{y \in \mathcal{N}_{x}} \mathcal{P}(x,y)  = M. 
\end{equation*}
In particular,
all equalities hold,
and hence $\mathcal{N}_{x} \subset \Omega$.

We now prove $V= \Omega$.
Fix $x\in V$.
We can take $x_0\in \Omega$ since $\Omega \neq \emptyset$.
We also take a minimal geodesic $\{ x_ i \}_{i=0}^l$ from $x_0$ to $x$.
From the above argument,
we see $x_1\in \Omega$.
Furthermore,
by induction,
we conclude $x_l\in \Omega$.
Thus $V=\Omega$, and hence $f\equiv M$.
\end{proof}

We next deduce the following superharmonicity:
\begin{lem}\label{lem:harmonic}
Under the same setting as in Theorem \ref{thm:Cheng},
for any $x,y\in M$ with $d(x,y)=\diam V$,
the function $\rho_x+\olarrow{\rho}_y$ is superharmonic.
\end{lem}
\begin{proof}
In view of Remark \ref{rem:pole},
\begin{equation*}
Kd(x,y)=\mathcal{H}(x,y).
\end{equation*}
This together with Theorem \ref{thm:Laplacian comparison} and the triangle inequality yields
\begin{equation}\label{eq:proofharm}
\mathcal{L}(\rho_x+\olarrow{\rho}_y)\geq K(\rho_x+\olarrow{\rho}_y)+(\mathcal{H}_x+\olarrow{\mathcal{H}}_{y})
                                           \geq Kd(x,y)-\mathcal{H}(x,y)=0.
\end{equation}
We complete the proof.
\end{proof}

We further show the following:
\begin{lem}\label{lem:subconstant}
Let $x,y\in V$ with $x \neq y$,
and let $\{x_{i}\}^{l}_{i=0}$ be a minimal geodesic from $x$ to $y$ with $l=d(x,y)$.
Set $z:=x_{a}$ and $w:=x_{b}$ for some $a,b \in \{0,\dots,l\}$ with $a < b$.
If $\kappa(x,y)=K$,
then $\kappa(z,w)=K$.
\end{lem}
\begin{proof}
In virtue of Proposition \ref{prop:local to global},
it is enough to prove that if $\kappa(x,y)\leq K$, then $\kappa(z,w)\leq K$.
We only consider the case where $a\geq 1$ and $b\leq l-1$.
By \eqref{eq:chain1} in Lemma \ref{lem:chain},
\begin{align*}
K d(x,y)&\geq \kappa(x,y) d(x,y)\geq \sum^{a-1}_{i=0}\kappa(x_{i},x_{i+1})+\kappa(z,w)d(z,w)+\sum^{l-1}_{i=b}\kappa(x_{i},x_{i+1})\\
              &\geq K(l-b+a)+\kappa(z,w)d(z,w)= Kd(x,y)-Kd(z,w)+\kappa(z,w)d(z,w).
\end{align*}
Hence, $\kappa(z,w)\leq K$.
In the other cases,
it suffices to use \eqref{eq:chain2}, \eqref{eq:chain3} instead of \eqref{eq:chain1}.
\end{proof}

We are now in a position to prove Theorem \ref{thm:Cheng}.

\begin{proof}[Proof of Theorem \ref{thm:Cheng}]
Assume that $K>0$ and \eqref{eq:equalBM}.
Lemmas \ref{lem:minimum principle} and \ref{lem:harmonic} lead us that
\begin{equation}\label{eq:rigid}
\rho_x+\olarrow{\rho}_y \equiv d(x,y)
\end{equation}
on $V$, and we complete the proof of the first claim of Theorem \ref{thm:Cheng}.
The second one directly follows from Remark \ref{rem:pole} and Lemma \ref{lem:subconstant}.
Also,
the third one can be derived from Remark \ref{rem:pole}.
Thus $(V,\mu)$ is spherically suspended.
\end{proof}

%%%%%%%%%%%%%%%%%%%%%%%%%%%%
\subsection{Sharpness of comparison geometric results}\label{sec:Sharpness}
Here we show that
under the same setting as in Theorem \ref{thm:Cheng},
the equalities in other comparison geometric results hold.
First,
we mention the Laplacian comparison:
\begin{thm}\label{thm:rigidLapcomp}
Under the same setting as in Theorem \ref{thm:Cheng},
the equalities in Theorem \ref{thm:Laplacian comparison} hold.
More precisely,
on $V$,
\begin{equation*}
\mathcal{L} \rho_{x}= K \rho_{x}+\mathcal{H}_x,\quad \mathcal{L}\olarrow{\rho}_y= K\olarrow{\rho}_y+\olarrow{\mathcal{H}}_{y}.
\end{equation*}
\end{thm}
\begin{proof}
By \eqref{eq:rigid},
the equalities in \eqref{eq:proofharm} hold.
We arrive at the desired assertion.
\end{proof}

We next discuss the eigenvalue comparison.
\begin{thm}\label{thm:rigideigcomp}
Under the same setting as in Theorem \ref{thm:Cheng},
the equality in Theorem \ref{thm:Lic} holds.
More precisely,
\begin{equation*}
\lambda_1(V)=K.
\end{equation*}
\end{thm}
\begin{proof}
We define a function $f:V\to \mathbb{R}$ by
\begin{equation*}
f:=\rho_x+\frac{\mathcal{H}_x}{K}.
\end{equation*}
Due to Theorem \ref{thm:rigidLapcomp},
\begin{equation*}
\mathcal{L}f=\mathcal{L}\rho_x=K\rho_x+\mathcal{H}_x=Kf.
\end{equation*}
In particular,
$K$ is an eigenvalue of $\mathcal{L}$ with eigenfunction $f$.
We obtain $\lambda_1(V)\leq K$,
and the equality in Theorem \ref{thm:Lic} holds.
\end{proof}

Cushing et al. \cite{CKKLMP} have proved Theorem \ref{thm:rigideigcomp} for unweighted, undirected regular graphs (see \cite[Theorem 1.5]{CKKLMP}).

%%%%%%%%%%%%%%%%%%%%%%%%%%%%
%%%%%%%%%%%%%%%%%%%%%%%%%%%%
%%%%%%%%%%%%%%%%%%%%%%%%%%%%
\section{Examples}\label{sec:Examples}

In this last section,
we aim to present (purely) directed graphs having maximal diameter.
For $k\in \mathbb{R}$,
we say that
$(V,\mu)$ has \textit{constant Ricci curvature $k$} if $\kappa(x,y)=k$ for all edges $(x, y) \in E$.
In that case
we denote by $\kappa(V,\mu)=k$.

%%%%%%%%%%%%%%%%%%%%%%%%%%%%
\subsection{Directed complete graphs}\label{sec:Directed complete}

For $n\geq 3$,
we consider the unweighted directed complete graph with $n$ vertices,
denoted by $\mathcal{K}_n$ (see Figure \ref{fig:complete}).

\begin{figure}[http]
        \begin{center} 
          \includegraphics[scale=0.8]{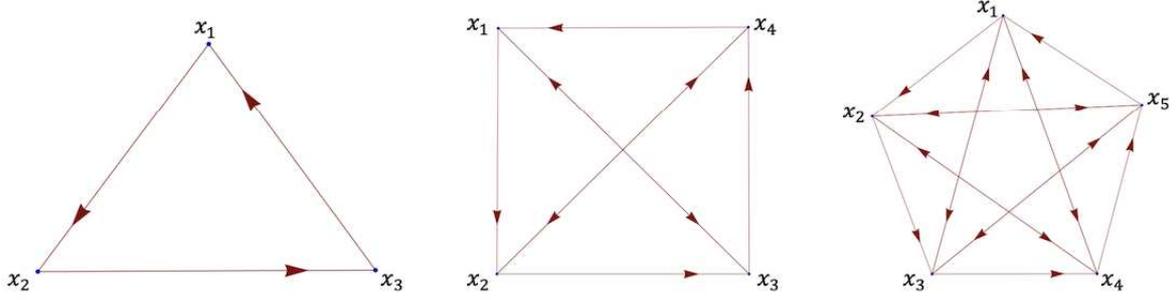}
          \caption{Directed complete graphs}
          \label{fig:complete}
        \end{center}
\end{figure}

It is easy to see that
\begin{equation}\label{eq:completediammean}
\diam \mathcal{K}_n=2,\quad \mathcal{H}_{x_{i}}=\olarrow{\mathcal{H}}_{x_{i}}=-\left(  1+\cfrac{1}{2 (n - 2)}  \right),\quad \Lambda=2+\cfrac{1}{n - 2}
\end{equation}
for all $n\geq 3$ and $i=1,\dots,n$.
The authors \cite{OSY} have calculated the Ricci curvature of the edges of $\mathcal{K}_n$ as follows (see \cite[Example 4.1]{OSY}):
(1) $\kappa(\mathcal{K}_3)=3/2$;
(2) if $n=4$,
then we have
\begin{align*}
\kappa (x_1, x_i) = \begin{cases}
		1 & \text{if $i = 2$},\\
		\cfrac{3}{2} & \text{if $i=3$}; 
		\end{cases}
\end{align*}
(3) if $n=5$,
then we have
\begin{align*}
\kappa (x_1, x_i) = \begin{cases}
		1 & \text{if $i = 2$},\\
		\cfrac{7}{6} & \text{if $i=3$ or $4$}; 
		\end{cases}
\end{align*}
(4) if $n \geq 6$,
then we have
\begin{align*}
\kappa (x_1, x_i) = \begin{cases}
		1 & \text{if $i = 2$ or $i\in \{4, \dots, n-2\}$},\\
		1+\cfrac{1}{2 (n - 2)} & \text{if $i=3$ or $n-1$}.
		\end{cases}
\end{align*}
In particular,
\begin{equation}\label{eq:completeRicci}
K = \begin{cases}
		\cfrac{3}{2} & \text{if $n=3$},\\
		       1                & \text{if $n\geq 4$}.
		\end{cases}
\end{equation}
By \eqref{eq:completediammean}, \eqref{eq:completeRicci},
the equality \eqref{eq:equalBM} in holds on $\mathcal{K}_n$ if and only if $n=3$.
Actually,
$\mathcal{K}_n$ is not spherically suspended for $n\geq 4$ although it is covered by minimal geodesics from $x_1$ to $x_n$.

\begin{rem}
It is remarkable that
undirected (usual) complete graph with $3$ vertices does not have maximal diameter.
\end{rem}

%%%%%%%%%%%%%%%%%%%%%%%%%%%%
\subsection{Directed triforce graphs}\label{sec:triforce}

We next observe the unweighted directed triforce graph $\mathcal{T}$ (see Figure \ref{fig:triforce}).

\begin{figure}[http]
        \begin{center} 
          \includegraphics[scale=0.5]{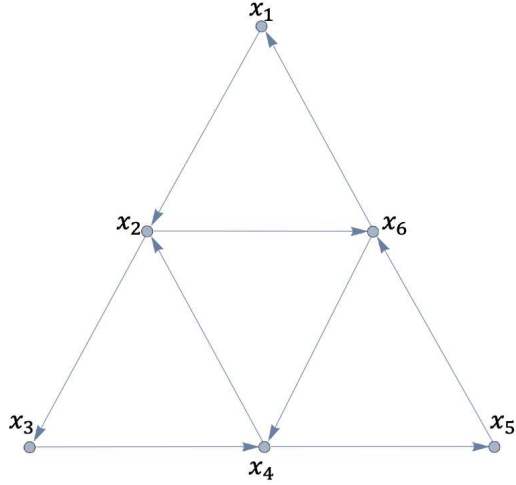}
          \caption{Directed triforce graph}
          \label{fig:triforce}
        \end{center}
\end{figure}
We verify that $\kappa(\mathcal{T})=3/4$,
and $\mathcal{T}$ has maximal diameter as follows:
It is trivial that
\begin{equation}\label{eq:diamtrif}
\diam \mathcal{T}=4.
\end{equation}
We calculate the asymptotic mean curvature.
Since $\mathcal{T}$ is Eulerian,
the formula \eqref{eq:Euler} implies 
\begin{equation}\label{eq:triforceP}
\mathcal{P}(x_1,z)= \begin{cases}
		\cfrac{1}{2} & \text{if $z=x_2$ or $x_6$},\\
		0 & \text{otherwise},
		\end{cases}\quad
		\mathcal{P}(x_2,z)= \begin{cases}
		\cfrac{1}{4} & \text{if $z=x_1$ or $x_3$ or $x_4$ or $x_6$},\\
		0 & \text{otherwise},
		\end{cases}
\end{equation}
for instance.
From straightforward computations,
it follows that
\begin{equation}\label{eq:lambdatrif}
\mathcal{H}_{x_i}=\olarrow{\mathcal{H}}_{x_i}=-\frac{3}{2},\quad \Lambda=3
\end{equation}
for all $i$.
Let us check that $\kappa(\mathcal{T})=3/4$.
To do so,
by the symmetry,
it is enough to calculate $\kappa(x_1,x_2),\,\kappa(x_2,x_6)$ and $\kappa(x_2,x_3)$.
We here only present the calculation for $\kappa(x_1,x_2)$,
and the others are left to the reader.
By \eqref{eq:triforceP},
\begin{align*}
\nu^\epsilon_{x_1}(z) = \begin{cases}
		1 - \epsilon & \text{if $z= x$,}\\
		\cfrac{\epsilon}{2} & \text{if $z=x_2$ or $x_6$,}\\
		0 & \text{otherwise},
		\end{cases}\quad 
\nu^\epsilon_{x_2}(z) = \begin{cases}
		1 - \epsilon & \text{if $z= y$,}\\
		\cfrac{\epsilon}{4} & \text{if $z=x_1$ or $x_3$ or $x_4$ or $x_6$,}\\
		0 & \text{otherwise}.
		\end{cases}
\end{align*}
We define a coupling $\pi$ of $(\nu^{\epsilon}_{x_1},\nu^{\epsilon}_{x_2})$ by
\begin{align*}
\pi(z,w):=\begin{cases}
		1 - \epsilon-\cfrac{\epsilon}{2} & \text{if $(z,w)=(x_1,x_2)$},\\
		\cfrac{\epsilon}{4} & \text{if $(z,w)=(x_1,x_1)$ or $(x_1,x_3)$ or $(x_6,x_4)$ or $(x_6,x_{6})$},\\
		0 & \text{otherwise}. 
		\end{cases}
\end{align*}
Then one can prove $W(\nu^{\epsilon}_{x_1},\nu^{\epsilon}_{x_2})\leq 1-3\epsilon/4$ by \eqref{eq:Wasserstein distance},
and hence $\kappa(x_1,x_2)\geq 3/4$.
To check the opposite inequality,
we define a $1$-Lipschitz function $f:V\to \mathbb{R}$ by
\begin{align*}
f(z):=\begin{cases}
		2 & \text{if $z = x_3$},\\
		1 & \text{if $z=x_2$},\\
		-1 & \text{if $z = x_6$},\\
		0 & \text{otherwise}.
		\end{cases}
\end{align*}
Applying Proposition \ref{prop:Kantorovich duality} to the function $f$,
we obtain $\kappa(x_1,x_2)\leq 3/4$.
This proves $\kappa(x_1,x_2)=3/4$.
Notice that
in the verification of $\kappa(x_2,x_6)=3/4$,
we use
\begin{align*}
\pi(z,w)&:=\begin{cases}
		1 - \epsilon-\cfrac{\epsilon}{4} & \text{if $(z,w)=(x_2,x_6)$},\\
		\cfrac{\epsilon}{4} & \text{if $(z,w)=(x_1,x_1)$ or $(x_2,x_2)$ or $(x_3,x_5)$ or $(x_4,x_4)$ or $(x_6,x_{6})$},\\
		0 & \text{otherwise},
		\end{cases}\\
		f(z)&:=\begin{cases}
		2 & \text{if $z = x_5$},\\
		1 & \text{if $z=x_4$ or $x_6$},\\
		0 & \text{otherwise}.
		\end{cases}
\end{align*}
Also,
in the verification of $\kappa(x_2,x_3)=3/4$,
\begin{align*}
\pi(z,w)&:=\begin{cases}
		1 - \epsilon-\cfrac{\epsilon}{2} & \text{if $(z,w)=(x_2,x_3)$},\\
		\cfrac{\epsilon}{4} & \text{if $(z,w)=(x_1,x_3)$ or $(x_3,x_3)$ or $(x_4,x_4)$ or $(x_6,x_4)$},\\
		0 & \text{otherwise},
		\end{cases}\\
		f(z)&:=\begin{cases}
		2 & \text{if $z = x_4$},\\
		1 & \text{if $z=x_3$ or $x_6$},\\
		0 & \text{otherwise}.
		\end{cases}
\end{align*}
Thus we arrive at
\begin{equation}\label{eq:Ricctrif}
\kappa(\mathcal{T})=\frac{3}{4},\quad K=\frac{3}{4}.
\end{equation}
Combining \eqref{eq:diamtrif}, \eqref{eq:lambdatrif} and \eqref{eq:Ricctrif},
we conclude that $\mathcal{T}$ has maximal diameter.

\begin{rem}
We consider the unweighted directed multi triforce graph (see Figure \ref{fig:multitriforce}).
\begin{figure}[http]
        \begin{center} 
          \includegraphics[scale=0.5]{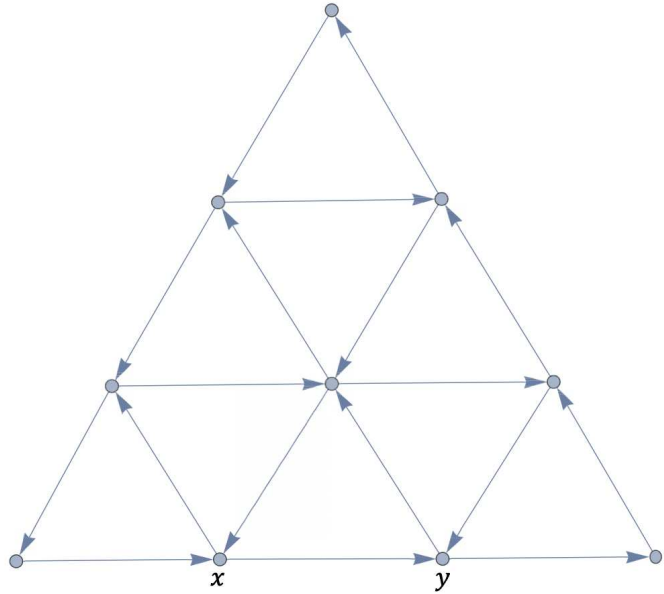}
          \caption{Directed multi triforce graph}
          \label{fig:multitriforce}
        \end{center}
\end{figure}
Then we see $\kappa(x,y)=0$;
in particular,
it does not have positive Ricci curvature.
\end{rem}

%%%%%%%%%%%%%%%%%%%%%%%%%%%%
\subsection{Product and maximal diameter}\label{sec:Product and maximal diameter}

We consider two simple, strongly connected, finite weighted directed graphs $(V,\mu)$ and $(V',\mu')$,
and also 
their $(\alpha,\beta)$-weighted Cartesian product $(V, \mu) \square_{(\alpha,\beta)} (V', \mu')$.
Let us denote by $D,\,D'$ and $D_{(\alpha,\beta)}$ the diameter of $(V,\mu),\,(V',\mu')$ and $(V, \mu) \square_{(\alpha,\beta)} (V', \mu')$,
respectively.
We also denote by $\Lambda,\,\Lambda'$ and $\Lambda_{(\alpha,\beta)}$ the value defined as \eqref{eq:supconst} of $(V,\mu),\,(V',\mu')$ and $(V, \mu) \square_{(\alpha,\beta)} (V', \mu')$,
respectively.
Furthermore,
We denote by $K,\,K'$ and $K_{(\alpha,\beta)}$ the value done by \eqref{eq:infconst} of $(V,\mu),\,(V',\mu')$ and $(V, \mu) \square_{(\alpha,\beta)} (V', \mu')$,
respectively.

\begin{lem}\label{lem:advproduct formula}
\begin{equation*}
D_{(\alpha,\beta)}=D+D',\quad \Lambda_{(\alpha,\beta)}=\frac{\beta}{\alpha+\beta}\Lambda+\frac{\alpha}{\alpha+\beta}\Lambda',\quad K_{(\alpha,\beta)}=\min \left\{ \frac{\beta}{\alpha+\beta}K, \frac{\alpha}{\alpha+\beta}K'  \right\}.
\end{equation*}
\end{lem}
\begin{proof}
The assertions for $D$ and $\Lambda$ are immediately derived from \eqref{eq:product distance} and \eqref{eq:product mean curv}.
For $K$,
Theorem \ref{thm:product Ricci} implies
\begin{align*}
K_{(\alpha,\beta)}&=\min \left\{ \inf_{x\rarrow y,\,x'\in V'}\kappa_{(\alpha,\beta)}((x,x'),(y,x')), \inf_{x\in V,\,x'\rarrow y'}\kappa_{(\alpha,\beta)}((x,x'),(x,y'))   \right\}\\
&=\min \left\{ \frac{\beta}{\alpha+\beta}\inf_{x\rarrow y}\kappa(x,y), \frac{\alpha}{\alpha+\beta}\inf_{x'\rarrow y'}\kappa'(x',y')   \right\}=\min \left\{ \frac{\beta}{\alpha+\beta}K, \frac{\alpha}{\alpha+\beta}K'  \right\}.
\end{align*}
This proves the lemma.
\end{proof}

The following claim together with the observation in Subsections \ref{sec:Directed complete} and \ref{sec:triforce} enables us to produce new directed graphs having maximal diameter:
\begin{thm}\label{thm:maxdiamproduct}
We assume $K,K'>0$.
Then the following are equivalent:
\begin{enumerate}\setlength{\itemsep}{+0.7mm}
\item The equality \eqref{eq:equalBM} holds on $(V,\mu)$ and $(V', \mu')$, and $\alpha D \Lambda'=\beta D' \Lambda$;\label{enum:preproduct}
\item the equality \eqref{eq:equalBM} holds on $(V, \mu) \square_{(\alpha,\beta)} (V', \mu')$.\label{enum:afterproduct}
\end{enumerate}
\end{thm}
\begin{proof}
The implication from (\ref{enum:preproduct}) to (\ref{enum:afterproduct}) is a direct consequence of Lemma \ref{lem:advproduct formula}.
Let us prove the opposite direction.
By virtue of Lemma \ref{lem:advproduct formula}, the assumption, and Theorem \ref{thm:BM},
\begin{equation}\label{eq:productpf1}
\frac{\displaystyle\frac{\beta}{\alpha+\beta}\Lambda+\frac{\alpha}{\alpha+\beta}\Lambda'}{D+D'}
=\frac{\Lambda_{(\alpha,\beta)}}{D_{(\alpha,\beta)}}=K_{(\alpha,\beta)}\leq \frac{\beta}{\alpha+\beta}K\leq \frac{\beta}{\alpha+\beta}\frac{\Lambda}{D},
\end{equation}
and hence $\alpha D \Lambda' \leq \beta D' \Lambda$.
Similarly,
\begin{equation}\label{eq:productpf2}
\frac{\displaystyle\frac{\beta}{\alpha+\beta}\Lambda+\frac{\alpha}{\alpha+\beta}\Lambda'}{D+D'}
=\frac{\Lambda_{(\alpha,\beta)}}{D_{(\alpha,\beta)}}=K_{(\alpha,\beta)}\leq \frac{\alpha}{\alpha+\beta}K'\leq \frac{\alpha}{\alpha+\beta}\frac{\Lambda'}{D'},
\end{equation}
and hence $\alpha D \Lambda' \geq \beta D' \Lambda$.
It follows that $\alpha D \Lambda' = \beta D' \Lambda$.
Furthermore,
the inequalities in \eqref{eq:productpf1}, \eqref{eq:productpf2} become equalities.
Thus we conclude the desired assertion.
\end{proof}

Cushing et al. \cite{CKKLMP} have proved Theorem \ref{thm:maxdiamproduct} for unweighted, undirected regular graphs (see \cite[Theorem 3.2]{CKKLMP}).

\begin{ex}
For example,
for $\alpha_1,\dots,\alpha_m>0$, the Cartesian product
\begin{equation*}
\{(\mathcal{K}_3\square_{(\alpha_1,\alpha_1)} \mathcal{K}_3)\square_{(\alpha_2,2\alpha_2)}\mathcal{K}_3\}\square_{(\alpha_3,3\alpha_3)} \cdots \square_{(\alpha_m,m\alpha_m)}\mathcal{K}_3
\end{equation*}
has maximal diameter,
whose diameter is $2(m+1)$, value defined as \eqref{eq:supconst} is $3$, and value defined as \eqref{eq:infconst} is $3/2(m+1)$.
\end{ex}

%%%%%%%%%%%%%%%%%%%%%%%%%%%
%%%%%%%%%%%%%%%%%%%%%%%%%%%
%%%%%%%%%%%%%%%%%%%%%%%%%%%
\subsection*{{\rm Acknowledgements}}
The authors thank Professor Takashi Shioya for informing them of \cite{M}.
They are also grateful to Doctor Daisuke Kazukawa for his useful comments.
The first named author was supported in part by JSPS KAKENHI (19K14532).
The first and second named authors were supported in part by JSPS Grant-in-Aid for Scientific Research on Innovative Areas ``Discrete Geometric Analysis for Materials Design" (17H06460).
The third named author was supported in part by JSPS KAKENHI (19K23411).

%%%%%%%%%%%%%%%%%%%%%%%%%%%%

\end{document}